\documentclass[english]{amsart}
\usepackage{amsfonts}
\usepackage{amssymb}
\usepackage{babel}
\usepackage{amsthm}
\usepackage{latexsym,amsmath}
\usepackage{mathrsfs}
\usepackage[numbers]{natbib}
\usepackage{enumerate}
\usepackage[utf8x]{inputenc}
\usepackage{underscore}
\usepackage{hyperref}
\usepackage{float}
\newtheorem{theorem}{Theorem}[section]
\newtheorem{lemma}[theorem]{Lemma}
\newtheorem{corollary}[theorem]{Corollary}
\newtheorem{proposition}[theorem]{Proposition}
\newtheorem{definition}[theorem]{Definition}

\DeclareMathOperator{\diag}{diag}

\DeclareMathOperator{\rank}{rank}

\DeclareMathOperator{\dist}{dist}

\DeclareMathOperator{\tri}{tri}
\DeclareMathOperator{\supp}{supp}

\title[Two-by-two upper triangular matrices and Morrey's conjecture]{Two-by-two upper triangular matrices and Morrey's conjecture}
\author[T.L.J. Harris]{Terence L. J. Harris}
\address{Department of Mathematics, University of Illinois, Urbana, IL 61801, U.S.A.}
\email{terence2@illinois.edu}
\author[B. Kirchheim]{Bernd Kirchheim}
\address{Mathematisches Institut, Universität Leipzig, Augustusplatz 10,
04109 Leipzig, Germany}
\email{bernd.kirchheim@math.uni-leipzig.de}
\author[C.-C. Lin]{Chun-Chi Lin}
\address{Department of Mathematics, National Taiwan Normal University, Taipei, 116 Taiwan}
\email{chunlin@math.ntnu.edu.tw}
\thanks{{\it Maths Subject Classification (2000):} 49J45}
\numberwithin{equation}{section}
\begin{document}
\begin{abstract} It is shown that every homogeneous gradient Young measure supported on matrices of the form $\begin{pmatrix} a_{1,1} & \cdots & a_{1,n-1} & a_{1,n} \\ 0 & \cdots & 0 & a_{2,n} \end{pmatrix}$ is a laminate. This is used to prove the same result on the 3-dimensional nonlinear submanifold of $\mathbb{M}^{2 \times 2}$ defined by $\det X = 0$ and $X_{12}>0$. \end{abstract}
\maketitle
\section{Introduction and preliminaries}\label{s:1} \begin{sloppypar}
\label{intro}  Let $\mathbb{M}^{m \times n}$ be the space of $m \times n$ matrices with real entries. A function \mbox{$f: \mathbb{M}^{m \times n} \to \mathbb{R}$} is rank-one convex if 
\[ f(\lambda X+ (1-\lambda)Y ) \leq \lambda f(X) + (1-\lambda )f(Y) \quad \text{ for all } \lambda \in [0,1], \]
for all $X,Y \in \mathbb{M}^{m \times n}$ with $\rank(X-Y) \leq 1$. A locally bounded Borel measurable function $f: \mathbb{M}^{m \times n} \to \mathbb{R}$  is quasiconvex if for every bounded domain $\Omega \subseteq \mathbb{R}^n$ and $X_0 \in \mathbb{M}^{m \times n}$,
\[ f(X_0) m\left(\Omega\right) \leq \int_{\Omega} f\left(X_0+ \nabla \phi(x) \right) \: dx, \]
for every $\phi \in C_0^{\infty}(\Omega, \mathbb{R}^m)$, where $\nabla \phi$ is the derivative of $\phi$. In this work we prove that rank-one convexity implies quasiconvexity on the $2 \times 2$ upper triangular matrices, and give some consequences and related results.

We just briefly mention the well known fact, that Morrey distilled the notion of quasiconvexity as the characterization of weak-$^*$ lower semicontinuity of variational integrals for vector valued problems, which ensures the general existence of minimizers. Rank-one convexity, on the other hand, is a much older notion and for smooth $f$ it is just equivalent to the Legendre-Hadamard condition, and hence a necessary condition for minimizers. So it comes as no surprise that quasiconvexity implies rank-one convexity of a general integrand $f$.

In 1952 Morrey conjectured that  the reverse implication fails, namely that rank-one convexity does not imply quasiconvexity \cite{morrey}. A counterexample for $m \geq 3$ and $n \geq 2$ was given by Šverák in \cite{sverak}, but the question remains open for $m =2$ and $n \geq 2$. Müller proved that rank-one convexity implies quasiconvexity on diagonal matrices (see \cite{muller}), which improved a result obtained earlier by Tartar \cite{tartar}. Other results toward a negative resolution of the conjecture on $\mathbb{M}^{2 \times 2}$ are given in \cite{chaudhuri,faraco,martin}. 

To explicitly formulate the problem on a subspace requires the dual notions of convexity for measures. Throughout, all probability measures are assumed to be Borel. A compactly supported probability measure $\mu$ on $\mathbb{M}^{m \times n}$ is called a laminate if 
\[ f\left( \overline{\mu} \right) \leq \int f \: d\mu \quad \text{ for all rank-one convex } f: \mathbb{M}^{m \times n} \to \mathbb{R}, \]
where $\overline{\mu} = \int X \: d\mu(X)$ is the barycentre of $\mu$. Similarly, $\mu$ is called a homogeneous gradient Young measure if the same inequality holds, but with rank-one convex replaced by quasiconvex (see \cite{kinderlehrer}). Polyconvexity can be defined analogously, but here we will only use the characterisation that a compactly supported probability measure $\mu$ on $\mathbb{M}^{2 \times 2}$ is polyconvex if $\int \det X \: d\mu(X) = \det(\overline{\mu})$. Let $\mathscr{M}_{pc}\left( U \right)$ be the set of polyconvex measures with compact support in a given set $U$, and define $\mathscr{M}_{qc}\left(U\right)$ and $\mathscr{M}_{rc}\left(U\right)$ similarly. 

 The question of whether rank-one convexity implies quasiconvexity is then equivalent to asking whether every homogeneous gradient Young measure is a laminate (see \cite{kinderlehrer}). In \cite{muller} Müller proved that every homogeneous gradient Young measure supported on the $2 \times 2$ diagonal matrices is a laminate. In \cite{lee} and \cite{muller} this was extended to the $n \times n$ diagonal matrices. The purpose of this work is to generalise the result for $2 \times 2$ diagonal matrices to the subspace
\[ \mathbb{M}^{2 \times n}_{\tri} := \left\{ \begin{pmatrix} a_{1,1} & \cdots & a_{1,n-1} & a_{1,n} \\ 0 & \cdots & 0 & a_{2,n} \end{pmatrix} \in \mathbb{M}^{2 \times n} \right\}. \]  When $n=2$, $\mathbb{M}^{2 \times 2}_{\tri}$ is the space of $2 \times 2$ upper triangular matrices. Up to linear isomorphisms preserving rank-one directions, the only other 3-dimensional subspace of $\mathbb{M}^{2 \times 2}$ is the symmetric matrices (see \cite[Corollary 6]{conti}). 

In Section \ref{s:4}, the result on upper-triangular matrices will be used to prove that rank-one convexity implies quasiconvexity on the 3-dimensional nonlinear manifold
\[ \{ X \in \mathbb{M}^{2 \times 2} : \det X = 0 \text{ and } X_{12}>0 \}. \]
This will be deduced as a corollary of a slightly more general result, which says that any Borel subset of $\{X_{12}>0\}$ has a ``dual set'' with the same convexity properties, and thus equality of rank-one convexity and quasiconvexity on one set implies equality of the two notions on the dual set. The method is similar to the one in \cite{chaudhuri}, where the same transformation is used to translate Müller's theorem on $2\times 2$ diagonal matrices to the set of symmetric matrices with $\det X = -1$ and $X_{11}>0$. 

\end{sloppypar}

\section{The linear space}\label{s:3}
In what follows, $P: \mathbb{M}^{2 \times n} \to \mathbb{M}^{2 \times n}$ will be the projection onto $\mathbb{M}^{2 \times n}_{\diag}$. Given a probability measure $\mu$ on $\mathbb{M}^{2 \times n}$, $P_{\#}\mu$ will denote the pushforward measure of $\mu$ by $P$, given by
\[ (P_{\#}\mu)(E) = \mu(P^{-1}(E)), \]
for any Borel set $E$. To show a homogeneous gradient Young measure $\mu$ supported in $\mathbb{M}^{2 \times n}_{\tri}$ is a laminate, the argument consists of two steps. The projection $P_{\#}\mu$ onto $\mathbb{M}^{2 \times n}_{\diag}$ is shown to be a gradient Young measure, and therefore a laminate by Müller's Theorem. It is then shown that since $P_{\#}\mu$ is a laminate, $\mu$ is also a laminate. The main reason that this method works is that $P$ is rank preserving, in the sense that $\rank X \leq 1$ if and only if $\rank(P(X)) \leq 1$, for any $X \in \mathbb{M}^{2 \times n}_{\tri}$. 

The proof requires a few extra definitions, which give a more constructive characterisation of laminates (see also \cite{pedregal}). 
\begin{definition} A set  $\{(t_1, Y_1), \ldots,(t_l,Y_l) \} \subseteq (0,1] \times \mathbb{M}^{m \times n}$ with $\sum_{i=1}^l t_i =1$ satisfies the $H_l$ condition if:
\begin{enumerate}[i)]
\item $l =2$ and $\rank(Y_1-Y_2) \leq 1$, or;
\item $l >2$ and after a permutation of the indices, $\rank(Y_1-Y_2) \leq 1$ and the set 
\[ \left\{ \left( t_1 + t_2, \frac{ t_1 Y_1 + t_2 Y_2}{t_1+t_2} \right), \left( t_3, Y_3 \right), \ldots ,\left(t_l, Y_l \right) \right\}, \]satisfies the $H_{l-1}$ condition. \end{enumerate}
\end{definition}
A convex combination of Dirac measures $\mu = \sum_{i=1}^N \lambda_i \delta_{X_i}$ is called a prelaminate of order (at most) $N$ if the set $\{(\lambda_1, X_1), \ldots, (\lambda_N, X_N)\}$ satisfies the $H_N$ condition. This definition essentially says that the class of prelaminates is the smallest class of probability measures that contains the Dirac masses and is closed under rank-one splitting of its atoms. We remark that in the definition, the representation of $\mu$ as a convex combination of Dirac measures is not necessarily unique and may have repeated $X_i$'s, which will require some care later. 

The following theorem is a special case of Theorem 4.12 in \cite{kirchheim}, see also Theorem 3.1 in \cite{mullersverak}. 
\begin{theorem} \label{seq} Let $\mu$ be a laminate with support inside a compact set $K \subseteq \mathbb{M}^{2 \times n}_{\diag}$, and let $U \subseteq \mathbb{M}^{2 \times n}_{\diag}$ be any relatively open neighbourhood of $K^{co}$. There exists a sequence $\mu^{(n)}$ of prelaminates supported in $U$, with common barycentre, such that $\mu^{(n)} \stackrel{*}{\rightharpoonup} \mu$. 
\end{theorem}
The following version of Müller's result on the space 
\[  \mathbb{M}^{2 \times n}_{\diag} := \left\{ \begin{pmatrix} a_{1,1} & \cdots & a_{1,n-1} & 0\\ 0 & \cdots & 0 & a_{2,n} \end{pmatrix} \in \mathbb{M}^{2 \times n} \right\}, \]
will also be needed; a proof is given in Appendix \ref{appendix}. 
\begin{theorem} \label{muller} Every homogeneous gradient Young measure supported in $\mathbb{M}^{2 \times n}_{\diag}$ is a laminate. \end{theorem}

\begin{sloppypar}
\begin{theorem} \label{uppertri} Every homogeneous gradient Young measure supported in $\mathbb{M}^{2 \times n}_{\tri}$ is a laminate. \end{theorem}
\begin{proof} The proof is divided into three parts.
\begin{enumerate}[(i)]
\item If $\mu$ is a homogeneous gradient Young measure supported in $\mathbb{M}^{2 \times n}_{\tri}$, then $P_{\#}\mu$ is a homogeneous gradient Young measure. 
\item If $\mu$ is a convex combination of Dirac measures supported in $\mathbb{M}^{2 \times n}_{\tri}$ and $P_{\#}\mu$ is a prelaminate, then $\mu$ is a prelaminate. 
\item If $\mu$ is a probability measure with compact support in $\mathbb{M}^{2 \times n}_{\tri}$ and $P_{\#}\mu$ is a laminate, then $\mu$ is a laminate.
\end{enumerate}
To prove (i), it is first shown that if $T: \mathbb{M}^{2 \times n} \to \mathbb{M}^{2 \times n}$ is defined by $T(X)=AXB$ where $A \in \mathbb{M}^{2 \times 2}$ and $B \in \mathbb{M}^{n \times n}$ is invertible, then $T_{\#} \nu$ is a homogeneous gradient Young measure whenever $\nu$ is a homogeneous gradient Young measure.

To show this, let $f: \mathbb{M}^{2 \times n} \to \mathbb{R}$ be a quasiconvex function and let \mbox{$g= f \circ T$}. Let $\Omega \subseteq \mathbb{R}^n$ be a nonempty bounded domain and let \mbox{$\phi \in C_0^{\infty}(\Omega, \mathbb{R}^2)$}. Define \mbox{$\psi \in C_0^{\infty}(B^{-1}(\Omega), \mathbb{R}^2)$} by $\psi(x) = A \phi(Bx)$, and let $X_0 \in \mathbb{M}^{2 \times n}$. Then \mbox{$\nabla \psi(x) = A \nabla \phi(Bx) B$}, and hence
\begin{align*} &\int_{\Omega} g\left(X_0 + \nabla \phi(x) \right) \: dx \\
= &\left|\det B\right| \int_{B^{-1}(\Omega)} f\left(AX_0B + \nabla \psi(y)\right) \: dy \\
\geq &\left|\det B\right|m\left(B^{-1}(\Omega) \right) f(AX_0B) \quad \text{ since $f$ is quasiconvex,} \\
= &m(\Omega) g(X_0). \end{align*}
This shows that $g$ is quasiconvex, which implies that
\[ \int f \: d(T_{\#} \nu) = \int g \: d\nu \geq g(\overline{\nu}) = f\left(T\left(\overline{\nu}\right) \right) =  f\left(\overline{ T_{\#} \nu} \right), \]
and therefore $T_{\#}\nu$ is a homogeneous gradient Young measure. 

Now let 
\[ A_k =  \begin{pmatrix} 1 & 0 \\ 0 & k \end{pmatrix}, \quad B_k = \begin{pmatrix}
  1 &&&  \\
& \ddots && \\
&& 1 & \\
&&& 1/k\end{pmatrix}, \]
and define $P^{(k)}: \mathbb{M}^{2 \times n} \to \mathbb{M}^{2 \times n}$ by $X \mapsto A_kXB_k$, so that for $X \in \mathbb{M}^{2 \times 2}_{\tri}$,
\[ P^{(k)}(X) =  \begin{pmatrix} x_{1,1} & \cdots & x_{1,n-1} & \frac{x_{1,n}}{k} \\
0 & \cdots & 0 & x_{2,n} \end{pmatrix}. \] 
Let $\mu$ be a (compactly supported) homogeneous gradient Young measure on $\mathbb{M}^{2 \times n}_{\tri}$. Then $P^{(k)} \to P$ uniformly on compact subsets of $\mathbb{M}^{2 \times n}_{\tri}$, and hence for any continuous function \mbox{$f: \mathbb{M}^{2 \times n} \to \mathbb{R}$},
\[ \lim_{k \to \infty} \int f \: d(P^{(k)}_{\#} \mu)   = \int f \: d(P_{\#} \mu). \]
Therefore  $P^{(k)}_{\#} \mu \stackrel{*}{\rightharpoonup} P_{\#} \mu$, and since the measures $P^{(k)}_{\#} \mu$ live on a common compact set, this shows that $P_{\#} \mu$ is a homogeneous gradient Young measure. 

For (ii), let $\mu$ be a finite convex combination of Dirac measures, such that $P_{\#} \mu$ is a prelaminate of order $N$. By induction on $N$ it may be assumed that a convex combination $\nu$ of Dirac measures on $\mathbb{M}^{2 \times n}_{\tri}$ is a prelaminate whenever $P_{\#} \nu$ is a prelaminate of order strictly less than $N$. Without loss of generality it may also be assumed that the projection $P$ is injective on the support of $\mu$, since $\mu$ can be obtained from its average along the $e_1 \otimes e_n$ direction by rank-one splitting in the $e_1 \otimes e_n$ direction.  

Therefore, by assumption $\mu$ can be written as $\mu = \sum_{i=1}^N \lambda_i \delta_{X_i}$ where the set $\{ ( \lambda_i,P(X_i),) : 1 \leq i \leq N\}$ satisfies the $H_N$ condition. This means that after a permutation of indices, \mbox{$\rank(P(X_1)-P(X_2)) \leq 1$} and the set 
\[ \left\{ \left( \lambda_1 + \lambda_2,P\left( \frac{ \lambda_1 X_1+\lambda_2X_2}{\lambda_1+\lambda_2} \right) \right),\left(\lambda_3,  P(X_3) \right), \ldots \left( \lambda_N,P(X_N) \right) \right\} \] satisfies the $H_{N-1}$ condition. Hence $\rank(X_1-X_2) \leq 1$, and by the inductive assumption the measure
\[ \nu = (\lambda_1 + \lambda_2) \delta_{ \frac{\lambda_1X_1}{\lambda_1+\lambda_2}+ \frac{\lambda_2X_2}{\lambda_1+\lambda_2}}+ \sum_{i=3}^N \lambda_i \delta_{X_i}, \]
is a prelaminate. The measure $\mu$ is a prelaminate since it can be obtained from $\nu$ by rank-one splitting. This proves (ii).

 Suppose $\mu$ satisfies the assumptions of (iii). Let $N$ be such that \mbox{$\supp \mu \subseteq B(0,N)$}, and identify $\mathbb{M}^{2 \times n}_{\tri}$ with $\mathbb{M}^{2 \times n}_{\diag} \times \mathbb{R}$. By disintegration (see Theorem 2.28 in \cite{ambrosio}) there exist probability measures $\lambda_X$ on $[-N,N]$ such that
\[ \int f\: d\mu = \int \int f(X,t) \: d\lambda_X(t) \: d(P_{\#}\mu)(X), \]
for all continuous $f: \mathbb{M}^{2 \times n}_{\tri} \to \mathbb{R}$, the inner integral being Borel measurable with respect to $X$. Therefore applying Lusin's Theorem (see \cite[Theorem 1.45]{ambrosio}) to the bounded function $X \mapsto \overline{\lambda_X}$ gives a uniformly bounded sequence of continuous functions $g^{(k)}(X) = (X,\psi^{(k)}(X))$ from $B(0,N) \cap \mathbb{M}^{2 \times n}_{\diag}$ into $B(0,N)\cap \mathbb{M}^{2 \times n}_{\tri}$ approximating $g(X) =(X, \overline{\lambda_X} )$, in the sense that 
\[ \lim_{k \to \infty} (P_{\# }\mu) \left\{ X \in \mathbb{M}^{2 \times n}_{\diag} : g^{(k)}(X) \neq g(X) \right\} =0. \] It follows that for any continuous function $f: \mathbb{M}^{2 \times n}_{\tri} \to \mathbb{R}$, 
\begin{equation} \label{lusin} \lim_{k \to \infty} \int (f \circ g^{(k)}) \: d(P_{\#}\mu) =  \int (f \circ g) \: d(P_{\#}\mu). \end{equation}
Using Theorem \ref{seq} let $\nu^{(j)}$ be a sequence of prelaminates supported in a common compact subset of $\mathbb{M}^{2 \times n}_{\diag}$ such that $\nu^{(j)} \stackrel{*}{\rightharpoonup} P_{\#}\mu$. If $f$ is rank-one convex, then
\begin{align*} \int f\: d\mu &= \int \int f(X,t) \: d\lambda_X(t) \: d(P_{\#}\mu)(X) \\
&\geq \int (f \circ g) \: d P_{\#}\mu \quad \text{ by convexity in the $e_1 \otimes e_n$ direction,}\\
&= \lim_{k \to \infty} \lim_{j \to \infty} \int f \: d g^{(k)}_{\#}\nu^{(j)} \quad \text{ by \eqref{lusin},} \\
&\geq \lim_{k \to \infty} \lim_{j \to \infty} f\left( \overline{g^{(k)}_{\#}\nu^{(j)}}\right) \quad \text{ by (ii),}\\
&= f(\overline{\mu}) \quad \text{ by \eqref{lusin}.}  \end{align*}
This shows that $\mu$ is a laminate, and therefore proves (iii).

Putting these together, if $\mu$ is a homogeneous gradient Young measure supported in $\mathbb{M}^{2 \times n}_{\tri}$, then $P_{\#}\mu$ is a homogeneous gradient Young measure by (i), and therefore a laminate by Theorem \ref{muller}. The fact that $P_{\#}\mu$ is a laminate then implies that $\mu$ is a laminate by (iii), and this proves the theorem. \end{proof} \end{sloppypar}
\section{The 3-dimensional nonlinear space}\label{s:4}
Let 
\[\mathbb{M}^{2 \times 2}_*  = \{ X \in \mathbb{M}^{2 \times 2} : \det X = 0 \}, \quad \mathbb{M}^{2 \times 2}_+ = \{ X \in \mathbb{M}^{2 \times 2} : X_{12} >0 \}. \]
The result of the previous section will be used to prove that every homogeneous gradient Young measure on \mbox{$\mathbb{M}^{2 \times 2}_* \cap \mathbb{M}^{2 \times 2}_+$} is a laminate. This is done via a change of variables that is often described as a partial Legendre transformation and was used in a similar context e.g. in 
\cite{chaudhuri,evans}. We will describe it now in detail, simplifying a few arguments from \cite{chaudhuri} along the way. 
\begin{sloppypar}
Given an open set $\Omega \subseteq \mathbb{R}^2$ and a smooth function $u=(u_1,u_2): \Omega \to \mathbb{R}^2$, consider the functions
\[ T_1(x) = (x_1, u_1(x)), \quad T_2(x)= (x_2,u_2(x)). \]
If $T_1$ is invertible with nonvanishing Jacobian, define the function \mbox{$v: T_1(\Omega) \to \mathbb{R}^2$} by $v \circ T_1 = T_2$, that is
\begin{equation} \label{uvdual} v(x_1, u_1(x)) = (x_2, u_2(x)) \quad \text{ for all } x \in \Omega. \end{equation}
This implies that 
\[ u(x_1, v_1(x)) = (x_2,v_2(x)) \quad \text{ for all } x \in T_1(\Omega). \]
which can be checked by substituting $x=T_1(y)$. If $u$ has gradient $X \in \mathbb{M}^{2 \times 2}$ at some point $x \in \Omega$, then by the chain rule $v$ has gradient 
\[ \Psi(X) = \frac{1}{X_{12}} \begin{pmatrix} -X_{11} & 1 \\ - \det X & X_{22} \end{pmatrix}, \]
at the point $T_1(x)$, where $\Psi$ is defined on $\mathbb{M}^{2 \times 2}_+$. If $S_1(x) = (x_1, v_1(x))$, then $S_1(T_1(x))=x$ by \eqref{uvdual}, and therefore $\Psi$ is a self-inverse mapping of $\mathbb{M}^{2 \times 2}_+$ onto itself. The main motivation for introducing this transformation is the easily verified fact that
\[ \Psi\left(\mathbb{M}^{2 \times 2}_* \cap \mathbb{M}^{2 \times 2}_+\right)= \mathbb{M}^{2 \times 2}_{\tri} \cap \mathbb{M}^{2 \times 2}_+. \]
\end{sloppypar}
Given a function $h:  \mathbb{M}^{2 \times 2}_+ \to \mathbb{R}$, define the dual function $\widetilde{h}:  \mathbb{M}^{2 \times 2}_+ \to \mathbb{R}$ by
\[ \widetilde{h}(X) = X_{12} h( \Psi(X)). \]
The term $X_{12}$ corresponds to the determinant of $\nabla T_1$, which will later simplify the change of variables in integration.

Given a probability measure $\mu$ on $\mathbb{M}^{2 \times 2}_+$, define the dual probability measure $\widetilde{\mu}$ on $\mathbb{M}^{2 \times 2}_+$ by 
\[ \int f \: d\widetilde{\mu}(X)  = \frac{1}{\overline{\mu}_{12}} \int \widetilde{f} \: d\mu. \]
for all Borel measurable $f: \mathbb{M}^{2 \times 2}_+ \to [0, \infty]$. The basic properties are summarised in the following proposition. 
\begin{proposition} \label{properties} Let $h : \mathbb{M}^{2 \times 2}_+ \to \mathbb{R}$ be a function, and let $\mu$ be a probability measure with compact support in $\mathbb{M}^{2 \times 2}_+$. Then:
\begin{enumerate}[(i)]
\item $\Psi = \Psi^{-1}$;
\item $\widetilde{\widetilde{h}} = h$;
\item $\widetilde{\widetilde{\mu}}=\mu$ and $\supp \widetilde{\mu} = \Psi(\supp \mu)$;
\item $\overline{ \widetilde{\mu} } = \Psi\left( \overline{\mu }\right)$ if and only if $\mu$ is polyconvex.
\end{enumerate}
\end{proposition}
\begin{sloppypar}
\begin{proof} Part (i) has been shown, and (ii) follows from (i). For (iii), the fact that $\supp \widetilde{\mu} = \Psi( \supp \mu)$ follows directly from the definition of $\widetilde{\mu}$ and the support. The barycentre of $\widetilde{\mu}$ is
\begin{equation} \label{barytilde} \overline{ \widetilde{\mu} } =  \int X \: d \widetilde{\mu} = \frac{1}{\overline{\mu}_{12} } \int X_{12} \Psi(X) \: d\mu(X) = \frac{1}{\overline{\mu}_{12} } \begin{pmatrix} -\overline{\mu}_{11} & 1 \\
- \int \det X \: d\mu(X) & \overline{\mu}_{22} \end{pmatrix}. \end{equation}
This shows that $\overline{\widetilde{\mu}}_{12} = \frac{1}{\overline{\mu}_{12}}$, and thus 
\[ \int f \: d\widetilde{ \widetilde{\mu }} = \overline{\mu}_{12} \int \widetilde{f} \: d\widetilde{ \mu } = \int \widetilde{\widetilde{f}} \: d\mu = \int f \: d\mu \quad \text{ by (ii).} \] 
Hence $\widetilde{\widetilde{\mu}}=\mu$, and therefore (iii) holds. For (iv), the measure $\mu$ is polyconvex if and only if $\int \det X \: d\mu(X) = \det( \overline{\mu})$, so combining this with \eqref{barytilde} finishes the proof. \end{proof}
Recall that $\mathscr{M}_{pc}\left(\mathbb{M}^{2 \times 2}_+\right)$ is the set of polyconvex measures with compact support in $\mathbb{M}^{2 \times 2}_+$, and $\mathscr{M}_{qc}\left(\mathbb{M}^{2 \times 2}_+\right)$, $\mathscr{M}_{rc}\left(\mathbb{M}^{2 \times 2}_+\right)$ are defined similarly. 
\end{sloppypar}
\begin{theorem} \label{dualmeasure} The function $\mu \mapsto \widetilde{\mu}$ maps $\mathscr{M}_{\Box}\left(\mathbb{M}^{2 \times 2}_+\right)$ bijectively onto itself, where $\Box \in \{pc,qc,rc\}$.
\end{theorem}
\begin{proof} If $\mu \in \mathscr{M}_{pc}\left(\mathbb{M}^{2 \times 2}_+\right)$ then by Proposition \ref{properties}
\[ \overline{ \widetilde{ \widetilde{\mu} } } = \overline{\mu} = \Psi\left( \Psi\left(\overline{\mu} \right) \right) = \Psi\left( \overline{ \widetilde{\mu} } \right), \]
and therefore $\widetilde{\mu} \in \mathscr{M}_{pc}\left(\mathbb{M}^{2 \times 2}_+\right)$. This proves the result in the case $\Box = pc$. 

If $\Box = qc$, the result will be proven for functions first, and then for measures. Let $h$ be quasiconvex on $\mathbb{M}^{2 \times 2}_+$, let $\Omega \subseteq \mathbb{R}^2$ be a nonempty bounded domain, and let $A \in \mathbb{M}^{2 \times 2}$, $\phi \in C_0^{\infty}(\Omega, \mathbb{R}^2)$ be such that the range of $A +\nabla \phi$ is contained in $\mathbb{M}^{2 \times 2}_+$. Then $A \in \mathbb{M}^{2 \times 2}_+$. Let $\psi(x) = Ax+ \phi(x)$, and write $\psi = (\psi_1, \psi_2)$. Then $\partial_2 \psi_1 $ is bounded below by a positive constant. This implies that $T_1$ is injective on $\Omega$, where
\[ T_1(x) = (x_1,\psi_1(x)), \quad T_2(x) = (x_2, \psi_2(x)). \]
The set $T_1(\Omega)$ is bounded since $\psi$ is Lipschitz, and $T_1$ is a diffeomorphism from $\Omega$ onto the bounded domain $T_1(\Omega)$ by the Inverse Function Theorem. Hence a Lipschitz map $g: T_1(\Omega) \to \mathbb{R}^2$ can be defined by $g \circ T_1 = T_2$. Let
\[ g_0(x) = g(x) -\Psi(A)x, \]
so that $g_0 \in C_0^{\infty}(T_1(\Omega), \mathbb{R}^2)$, and $(\nabla g)(T_1(x) ) = \Psi( \nabla \psi(x) )$. Then
\begin{align*} \int_{\Omega} \widetilde{h}(A + \nabla \phi(x) ) \: dx &= \int_{T_1(\Omega)}  h( \Psi(A) + \nabla g_0(x) ) \: dx   \\
&\geq m(T_1(\Omega)) h(\Psi(A))   \\
&= m(\Omega) \widetilde{h}(A), \end{align*}
where the last equality used the fact that $m(T_1(\Omega)) = A_{12} m(\Omega)$, which follows from an integration by parts or that the determinant is a null Lagrangian. This inequality shows that $\widetilde{h}$ is quasiconvex on $\mathbb{M}^{2 \times 2}_+$. 

By Theorem 1.6 in \cite{kirchheim}, a compactly supported probability measure $\mu$ on $\mathbb{M}^{2 \times 2}$ is a homogeneous gradient Young measure if and only if it satisfies Jensen's inequality for all quasiconvex $f: U \to \mathbb{R}$, where $U$ is any open neighbourhood of $(\supp \mu)^{co}$. Hence if $\mu \in \mathscr{M}_{qc}\left(\mathbb{M}^{2 \times 2}_+\right)$ and $f: \mathbb{M}^{2 \times 2}_+ \to \mathbb{R}$ is quasiconvex on $\mathbb{M}^{2 \times 2}_+$, then by Proposition \ref{properties},
\[ \int f \: d\widetilde{\mu} = \frac{1}{\overline{\mu}_{12}} \int \widetilde{f} \: d\mu \geq \frac{1}{\overline{\mu}_{12}} \widetilde{f}(\overline{\mu}) = f \left(\overline{ \widetilde{\mu} } \right), \]
and therefore $\widetilde{\mu} \in \mathscr{M}_{qc}\left(\mathbb{M}^{2 \times 2}_+\right)$. This covers the case $\Box=qc$. 

The case with $\Box = rc$ can also be done by duality; it suffices to show that $\widetilde{f}$ is rank-one convex whenever $f$ is. This follows from the $\Box=qc$ case since a function is rank-one convex if and only if it satisfies Jensen's inequality for all homogeneous gradient Young measures supported on two points, and this class of measures is mapped onto itself by $\mu \mapsto \widetilde{\mu}$. \end{proof}
\begin{sloppypar}
\begin{corollary} Every homogeneous gradient Young measure supported on \mbox{$\mathbb{M}^{2 \times 2}_* \cap \mathbb{M}^{2 \times 2}_+$}
is a laminate. 
\end{corollary}
\begin{proof} This follows from Theorem \ref{uppertri} and Theorem \ref{dualmeasure}, since if $\mu$ is a homogeneous gradient Young measure supported in $\mathbb{M}^{2 \times 2}_* \cap \mathbb{M}^{2 \times 2}_+$, then since $\Psi$ maps \mbox{$\mathbb{M}^{2 \times 2}_* \cap \mathbb{M}^{2 \times 2}_+$} onto $\mathbb{M}^{2 \times 2}_{\tri} \cap \mathbb{M}^{2 \times 2}_+$, the measure $\widetilde{\mu}$ is a homogeneous gradient Young measure supported in $\mathbb{M}^{2 \times 2}_{\tri} \cap \mathbb{M}^{2 \times 2}_+$. Therefore $\widetilde{\mu}$ is a laminate, and this implies that $\mu = \widetilde{ \widetilde{ \mu} }$ is a laminate. \end{proof}
\end{sloppypar}

\begin{appendix}
\section{Appendix: The diagonal case} \label{appendix}
\begin{sloppypar}
This section contains one particular generalisation of Müller's result from the $2 \times 2$ diagonal matrices $\mathbb{M}^{2 \times 2}_{\diag}$ to the subspace
\[ \mathbb{M}^{2 \times n}_{\diag} := \left\{ \begin{pmatrix} a_{1,1} & \cdots & a_{1,n-1} & 0 \\ 0 & \cdots & 0 & a_{2,n} \end{pmatrix} \in \mathbb{M}^{2 \times n} \right\}, \quad n \geq 2. \]
This is used to prove the result for $\mathbb{M}^{2 \times n}_{\tri}$. The proof has only minor modifications from the one in \cite{muller}, but is included for convenience. As in \cite{lee}, define the elements of the Haar system in $L^2(\mathbb{R}^n)$ by 
\[ h_Q^{(\epsilon)}(x) = \prod_{j=1}^n h_{I_j}^{\epsilon_j}(x_j), \quad \text{ for } x \in \mathbb{R}^n, \]
where $\epsilon \in \{0,1\}^n \setminus \{(0, \ldots, 0)\}$, $Q= I_1 \times \cdots \times I_n$ is a dyadic cube in $\mathbb{R}^n$, the $I_j$'s are dyadic intervals of equal size and the convention $0^0=0$ is assumed. A dyadic interval is always of the form $\left[k \cdot 2^{-j}, (k+1)\cdot 2^{-j} \right)$ with $j,k \in \mathbb{Z}$. For a dyadic interval $I=[a,b)$, $h_I$ is defined by 
\[ h_I(x) = h_{[0,1)}\left( \frac{ x-a}{b-a} \right) \quad \text{ for } x \in \mathbb{R}, \]
where 
\[ h_{[0,1)} = \chi_{\left[0, \frac{1}{2}\right)} - \chi_{\left[\frac{1}{2}, 1\right)}. \]
For $j \in \mathbb{Z}$ and $k \in \mathbb{Z}^n$, the notation $h^{(\epsilon)}_{j,k}=h^{(\epsilon)}_Q$ will be used, where
\[ Q =Q_{j,k} =  \bigg[ \frac{ k_1}{2^j},\frac{ k_1+ 1}{2^j} \bigg) \times \cdots \times \bigg[ \frac{ k_n}{2^j}, \frac{ k_n+ 1}{2^j} \bigg). \]
The standard basis vectors in $\mathbb{R}^n$ or $\{0,1\}^n$ will be denoted by $e_j$. The Riesz transform $R_j$ on $L^2(\mathbb{R}^n)$ is defined through multiplication on the Fourier side by $-i \xi_j/|\xi|$. In \cite[Theorem 2.1]{lee} and \cite[Theorem 5]{muller} it was shown that if $\epsilon \in \{0,1\}^n$ satisfies $\epsilon_j =1$, then there is a constant $C$ such that
\begin{equation} \label{interpolatory} \left\|P^{(\epsilon)}u \right\|_2 \leq C \|u\|_2^{1/2} \|R_j u \|_2^{1/2} \quad \text{ for all } u \in L^2(\mathbb{R}^n), \end{equation}
where $\epsilon$ is fixed and $P^{(\epsilon)}$ is the projection onto the closed span of the set
\[ \left\{ h_Q^{(\epsilon)} : Q \subseteq \mathbb{R}^n \text{ is a dyadic cube } \right\}. \]
\begin{lemma} \label{jensen} If $f: \mathbb{M}^{2 \times n} \to \mathbb{R}$ is rank-one convex with $f(0)=0$, and if \mbox{$u_1, \ldots, u_{n-1}, v_n$} have finite expansions in the Haar system
\[u_i = \sum_{ \epsilon_n =0} \sum_{j=J}^K \sum_{k \in \mathbb{Z}^n} a_{j,k, i}^{(\epsilon)} h_{j,k}^{(\epsilon)} \quad \text{ for } 1 \leq i \leq n-1, \text{ and } \quad  \quad v_n = \sum_{j=J}^K \sum_{k \in \mathbb{Z}^n} b_{j,k} h_{j,k}^{(e_n)}, \]
so that $a_{j,k, i}^{(\epsilon)}=b_{j,k}=0$ whenever $|k|$ is sufficiently large, then
\[ \int_{\mathbb{R}^n} f\begin{pmatrix} u_1 & \cdots & u_{n-1} & 0 \\
0 & \cdots & 0 & v_n \end{pmatrix} \: dx  \geq 0. \]
\end{lemma}
\begin{proof} The assumption that  $a_{j,k, i}^{(\epsilon)}=b_{j,k}=0$ for $|k|$ sufficiently large means the integral converges absolutely. Let
\[\widetilde{u}_i = \sum_{ \epsilon_n = 0} \sum_{j=J}^{K-1} \sum_{k \in \mathbb{Z}^n} a_{j,k, i}^{(\epsilon)} h_{j,k}^{(\epsilon)} \quad \text{ for } 1 \leq i \leq n-1, \text{ and let } \quad \widetilde{v}_n = \sum_{j=J}^{K-1} \sum_{k \in \mathbb{Z}^n} b_{j,k} h_{j,k}^{(e_n)}. \]
Then on $Q_{K,k}$, for any $k \in \mathbb{Z}^n$, 
\[ u_i' := u_i - \widetilde{u}_i =\sum_{ \epsilon_n = 0}  a_{K,k, i}^{(\epsilon)} h_{K,k}^{(\epsilon)}, \quad  v_n' := v_n -\widetilde{v}_n = b_{K,k}h_{K,k}^{(e_n)},  \] 
and
\begin{align*} &\int_{Q_{K,k}} f\begin{pmatrix} u_1 & \cdots & u_{n-1} & 0 \\
0 & \cdots & 0 & v_n \end{pmatrix} \: dx \\
= &\int_{Q_{K,k}} f\begin{pmatrix} \widetilde{u}_1+ u_1'  & \cdots & \widetilde{u}_{n-1}+  u_{n-1}' & 0 \\
0 & \cdots & 0 & \widetilde{v}_n + v_n' \end{pmatrix} \: dx_1 \: \cdots \: dx_n. \end{align*}
The bottom row is constant in $x_1, \ldots, x_{n-1}$ on $Q_{K,k}$, and so the function is convex for the integration with respect to $x_1, \ldots, x_{n-1}$. The terms $\widetilde{u}_i$ and $\widetilde{v}_n$ are constant on $Q_{K,k}$, and the $x_1, \ldots, x_{n-1}$ integral of $u_i'$ over the $(n-1)$-dimensional dyadic cube inside $Q_{K,k}$ is zero (for any $x_n$). Hence applying Jensen's inequality for convex functions gives
\begin{align*} &\int_{Q_{K,k}} f\begin{pmatrix} u_1 & \cdots & u_{n-1} & 0 \\
0 & \cdots & 0 & v_n \end{pmatrix} \: dx \\
\geq &\int_{Q_{K,k}} f\begin{pmatrix} \widetilde{u}_1 & \cdots & \widetilde{u}_{n-1} & 0 \\
0 & \cdots & 0 & \widetilde{v}_n + v_n' \end{pmatrix} \: dx_1 \: \cdots \: dx_n. \end{align*}
Applying Jensen's inequality similarly to the integration in $x_n$, and summing over all $k \in \mathbb{Z}^n$ gives
\[ \int_{\mathbb{R}^n} f\begin{pmatrix} u_1 & \cdots & u_{n-1} & 0 \\
0 & \cdots & 0 & v_n \end{pmatrix} \: dx  \geq \int_{\mathbb{R}^n} f\begin{pmatrix} \widetilde{u}_1 & \cdots & \widetilde{u}_{n-1} & 0 \\
0 & \cdots & 0 & \widetilde{v}_n \end{pmatrix} \: dx.\]
By induction this proves the lemma. \end{proof}
\begin{sloppypar}

\begin{proof}[Proof of Theorem \ref{muller}] Let $\mu$ be a homogeneous gradient Young measure supported in $\mathbb{M}^{2 \times n}_{\diag}$, and let $f: \mathbb{M}^{2 \times n} \to \mathbb{R}$ be a rank-one convex function. It is required to show that
\[ \int f \: d\mu \geq f( \overline{\mu}). \]
Without loss of generality it may be assumed that $\overline{\mu}=0$ and that $f(0)=0$.  After replacing $f$ by an extension of $f$ which is equal to $f$ on $(\supp \mu)^{co}$, it can also be assumed that there is a constant $C$ with
\begin{equation} \label{quadratic} |f(X)| \leq C(1+|X|^2) \quad \text{ for all } X \in \mathbb{M}^{2 \times n}. \end{equation}

Let $\Omega \subseteq \mathbb{R}^n$ be the open unit cube. By the characterisation of gradient Young measures \cite[Theorem 8.16]{pedregal} there is a sequence $\phi^{(j)}=(\phi^{(j)}_1,\phi^{(j)}_2)$ in $W^{1, \infty}(\Omega, \mathbb{R}^2)$ whose gradients generate $\mu$, which means that
\begin{equation} \label{young} \lim_{j \to \infty} \int_{\Omega} \eta(x) g\left( \nabla \phi^{(j)}(x) \right) \: dx = \int g \: d\mu \cdot \int_{\Omega} \eta(x) \: dx, \end{equation}
for any continuous $g$ and for all $\eta \in L^1(\Omega)$. In particular $\nabla \phi^{(j)} \to 0$ weakly in $L^2(\Omega, \mathbb{M}^{2 \times n})$. By the sharp version of the Zhang truncation theorem (see \cite[Corollary 3]{muller2}) it may be assumed that 
\begin{equation} \label{zhang} \left\|\dist\left(\nabla \phi^{(j)}, \mathbb{M}^{2 \times n}_{\diag} \right) \right\|_{\infty} \to 0. \end{equation} 
As in Lemma 8.3 of \cite{pedregal}, after multiplying the sequence by cutoff functions and diagonalising in such a way as to not affect \eqref{zhang}, it can additionally be assumed that $\phi^{(j)} \in W_0^{1, \infty}(\Omega, \mathbb{R}^2)$ (the choice $p=\infty$ is not that important, any large enough $p$ would work). 

Let $P_1: L^2(\mathbb{R}^n) \to L^2(\mathbb{R}^n)$ be the projection onto the closed span of
\[ \{ h_Q^{(\epsilon)} : Q \subseteq \mathbb{R}^n \text{ is a dyadic cube and } \epsilon_n = 0 \}, \]
and let \mbox{$P_2: L^2(\mathbb{R}^n) \to L^2(\mathbb{R}^n)$} be the projection onto the closed span of
\[ \{ h_Q^{(\epsilon)} : Q \subseteq \mathbb{R}^n \text{ is a dyadic cube and } \epsilon = e_n \}. \] 
Write $w^{(j)} = \nabla \phi^{(j)}$, so that by \eqref{zhang} and the fact that $R_i \partial_j \phi = R_j \partial_i \phi$, 
\[ \left\|R_n w^{(j)}_{1,1} \right\|_2, \ldots, \left\|R_nw^{(j)}_{1,n-1}\right\|_2 \to 0, \quad \left\|R_1w^{(j)}_{2,n}\right\|_2, \ldots, \left\|R_{n-1}w^{(j)}_{2,n}\right\|_2 \to 0. \]
Hence by \eqref{interpolatory} and orthogonality
\begin{equation} \label{projections} \left\|  w^{(j)}_{1,1}-P_1 w^{(j)}_{1,1} \right\|_2, \ldots ,\left\|  w^{(j)}_{1,n-1}-P_1 w^{(j)}_{1,n-1}\right\|_2 \to 0, \quad \left\| w^{(j)}_{2,n}-P_2 w^{(j)}_{2,n} \right\|_2 \to 0. \end{equation}

The function $f$ is separately convex since it is rank-one convex. Hence by the quadratic growth of $f$ in \eqref{quadratic} (see Observation 2.3 in \cite{matousek98}), there exists a constant $K$ such that
\begin{equation} \label{lipschitz} |f(X)-f(Y)| \leq K(1+|X|+|Y|)|X-Y| \text{ for all } X,Y \in \mathbb{M}^{2 \times n}. \end{equation}
Hence applying $\eqref{young}$ with $\eta= \chi_{\Omega}$ gives
\begin{align} \notag \int f \: d\mu  &= \lim_{j \to \infty} \int_{\Omega} f\left( w^{(j)} \right)\: dx \\
\label{liminf} &= \lim_{j \to \infty} \int_{\Omega} f\begin{pmatrix} P_1w^{(j)}_{11} & \cdots & P_1w^{(j)}_{1,n-1} & 0 \\
0 & \cdots & 0 & P_2w^{(j)}_{2,n} \end{pmatrix} \: dx, \end{align}
by \eqref{zhang}, \eqref{projections}, \eqref{lipschitz} and the Cauchy-Schwarz inequality. The functions $w^{(j)}$ are supported in $\overline{\Omega}$ and satisfy $\int_{\Omega} w^{(j)} \: dx = 0$ by the definition of weak derivative. Hence the $L^2(\mathbb{R}^n)$ inner product satisfies $\left\langle w^{(j)}, h_Q^{(\epsilon)} \right\rangle = 0$ whenever $Q$ is a dyadic cube not contained in $\overline{\Omega}$. This implies that $P_1w^{(j)}$ and $P_2w^{(j)}$ are supported in $\overline{\Omega}$. The integrand in \eqref{liminf} therefore vanishes outside $\overline{\Omega}$, and so
\[  \int f\: d\mu = \lim_{j \to \infty} \int_{\mathbb{R}^n} f\begin{pmatrix} P_1w^{(j)}_{11} & \cdots & P_1w^{(j)}_{1,n-1} & 0 \\
0 & \cdots & 0 & P_2w^{(j)}_{2,n} \end{pmatrix} \: dx \geq 0 \]
by \eqref{lipschitz} and Lemma \ref{jensen}. This finishes the proof. \end{proof} \end{sloppypar}
\end{sloppypar} \end{appendix}



%
%

\end{document}